\documentclass[11pt]{article}

\usepackage{epsfig}
\usepackage{epstopdf}
\usepackage[titletoc,title]{appendix}
\usepackage{graphicx}
\usepackage{latexsym, cite, bm, amsmath}
\usepackage{contour,soul,color}
\usepackage{booktabs,threeparttable} 
\usepackage[rotateright]{rotating} 
\usepackage{tabularx,multirow} 
\usepackage{cases}
\usepackage{float}
\usepackage{mathtools}
\usepackage{caption}
\usepackage{diagbox}
\usepackage{lscape}
\usepackage{amsfonts}
\usepackage{amsmath,url,cases}
\usepackage{amsthm}
\usepackage{indentfirst}

\usepackage{makecell}
\usepackage[english]{babel}
\usepackage[latin1]{inputenc}
\usepackage{times}
\usepackage{bbding}
\usepackage{latexsym}
\usepackage{epic}
\usepackage{indentfirst}
\usepackage{hyperref}
\usepackage{subfigure}
\usepackage{fancyvrb}
\usepackage{amssymb}
\usepackage{url}
\usepackage{framed}
\graphicspath{{figure/}}
\usepackage[boxed,ruled,commentsnumbered,titlenotnumbered]{algorithm2e}

\usepackage[left=1.10in,top=1.10in,right=1.10in,bottom=1.10in]{geometry}

\numberwithin{equation}{section}

\newtheorem{theorem}{Theorem}[section]

\newtheorem{lemma}[theorem]{Lemma}

\newtheorem{assumption}[theorem]{Assumption}
\newtheorem{remark}[theorem]{Remark}

\newtheorem{condition}{Condition}[section]

\numberwithin{equation}{section}

\def\be{\begin{equation}}
\def\ee{\end{equation}}

\def\la{\lambda}
\def \[{\begin{equation}}
\def \]{\end{equation}}

\def\R{{\mathbb R}}

\def\L{{\cal L}}

\def\nn{\nonumber}
\def\Tsf{\mathsf{T}}

\begin{document}

\title{ Alternating Direction Method of Multipliers with Variable Metric Indefinite Proximal Terms for Convex Optimization
}

\author{Yan Gu$^*$
and
Nobuo Yamashita$^*$
\\[0.2cm]
\small{$^*$Graduate School of Informatics, Kyoto University, Kyoto 6068501, Japan.}\\
\small{Email: \texttt{yan@amp.i.kyoto-u.ac.jp; nobuo@i.kyoto-u.ac.jp.}}\\
}
\date{\today}
\maketitle

\begin{abstract}
This paper studies a proximal alternating direction method of multipliers (ADMM) with variable metric indefinite proximal terms for linearly constrained convex optimization problems.
The proximal ADMM plays an important role in many application areas, since the subproblems of the method are easy to solve. Recently, it is reported that the proximal ADMM with a certain fixed indefinite proximal term is faster than that with a positive semidefinite term, and still has the global convergence property.
On the other hand, Gu and Yamashita studied a variable metric semidefinite proximal ADMM whose proximal term is generated by the BFGS update. They reported that a slightly indefinite matrix also makes the algorithm work well in their numerical experiments.
Motivated by this fact, we consider a variable metric indefinite proximal ADMM, and give sufficient conditions on the proximal terms for the global convergence.
Moreover, we propose a new indefinite proximal term based on the BFGS update which can satisfy the conditions for the global convergence.
\end{abstract}

\textbf{Keywords:}
alternating direction method of multipliers, variable metric indefinite proximal term, BFGS update, global convergence, convex optimization


\section{Introduction}\label{into}
We consider the following convex composite optimization problem:
\[\label{cp}
 \min \,  \left\{f(x) + g(y)\; | \;  Ax+By=b, \; x \in \R^{n}, \; y \in \R^{n}\right\},
\]
where $f\colon \R^{n} \rightarrow \R \cup \{\infty\}$ and $g\colon \R^{n} \rightarrow \R \cup \{\infty\}$ are proper convex functions, $A\in \R^{m \times n}, B\in \R^{m \times n}$ and $b \in \R^m$.
Various practical problems of science and engineering, such as machine learning \cite{koh2007interior, yin2008bregman}, total variation denoising \cite{rudin1992nonlinear} and statistics \cite{trevor2009elements} can be formulated as Problem \eqref{cp}.
Usually, we say that $f$ is a loss function and $g$ is a structured regularization term.


The augmented Lagrangian function of \eqref{cp} is defined as
 \[\label{augL}
\L_{\beta}(x,y,\la):= f(x) + g(y) -\langle \la,Ax+By-b \rangle+\frac{\beta}{2}\|Ax + By - b \|^2,\]
where $\lambda \in \R^m$ is the Lagrangian multiplier for the linear constraints $Ax + By = b$ in \eqref{cp}, and $\beta$ is a positive scalar. Note that $\L_{\beta} \colon \R^{n} \times \R^{n} \times \R^{m} \rightarrow \R$.

A number of efficient first-order algorithms have been developed for problem \eqref{cp} including operator splitting methods \cite{attouch2010parallel, bauschke2011convex, combettes2009iterative, douglas1956numerical, eckstein1992douglas, lions1979splitting}, gradient methods \cite{nesterov2013gradient, tseng2010approximation, tseng2009coordinate}, primal dual methods \cite{chambolle2011first, chen2013primal, esser2009general}, etc. One may solve problem \eqref{cp} is the classical augmented Lagrangian method (ALM), which generates the updates
\begin{equation}{\label{equ:ALM}}
\left\{
\begin{aligned}
& (x^{k+1}, y^{k+1}) = \arg\min_{x,y}\ \L_{\beta}(x,y,\la^k) \\
& \la^{k+1}  = \la^{k}-\beta(A x^{k+1} + By^{k+1} - b).
\end{aligned}
\right.
\end{equation}
In this case, the vectors $x^{k+1}$ and $y^{k+1}$ should be updated at the same time ignoring the separability of the original functions. Generally, the joint minimization problem \eqref{equ:ALM} is a challenge to be solved exactly or approximately with a high accuracy. We want to exploit the separability of the objective function to reduce the difficulty. The classical ADMM is one of such methods, and it efficiently solves problem \eqref{cp} \cite{glowinski1975approximation, gabay1976dual}. The convergence analysis for the classical ADMM can be referred to \cite{glowinski1975approximation, gabay1976dual, fortin1983chapter, boyd2011distributed, eckstein2015understanding}.

Fazel et al. \cite{fazel2013hankel} proposed a more convenient semi-proximal ADMM by adding proximal terms to subproblems which takes the following scheme:
\begin{subnumcases}{\label{equ:sADMM}}
\label{equ:sADMM1} x^{k+1}=\arg\min_{x}\ \L_{\beta}(x,y^k,\la^k) + \frac12 \|x - x^{k}\|_{S}^2,\\
\label{equ:sADMM2} y^{k+1}=\arg\min_{y}\ \L_{\beta}(x^{k+1},y,\la^k) + \frac12 \|y - y^{k}\|_{T}^2,\\
\label{equ:sADMM3} \la^{k+1}=\la^{k}-\alpha\beta(A x^{k+1} + By^{k+1} - b),
\end{subnumcases}
where $\alpha \in (0, (1+\sqrt 5)/2)$, and $S, T \succeq 0$. For a vector $z\in \R^n$ and a semidefinite matrix $G$, the norm $\|\cdot\|_G$ is defined by $\|z\|_{G} \;=\sqrt{z^ \top Gz}$. In this paper, even if $G \in \R^{n \times n}$ is not positive semidefinite, we denote $\|z\|^2_G = z^ \top Gz$ for simplicity.

The proximal ADMM covers the classical ADMM when $S=T=0$. When $S$ and $T$ are two positive definite matrices and $\alpha=1$, this semi-proximal ADMM reduces to the proximal ADMM proposed by Eckstein \cite{eckstein1994somesaddle}.
The proximal ADMM has an advantage that its subproblems are easy to solve, and it also can efficiently handle the multi-block convex optimization problem which is known as block-wise ADMM \cite{he2015block}. See \cite{deng2012global, he2012convergence, fazel2013hankel, xu2011class} for more details of the semi-proximal ADMM.

It is well known that the global convergence of the semi-proximal ADMM \eqref{equ:sADMM} is easier to prove. However, it is not satisfactory in numerical performance.
The paper \cite{Deng2016} mentioned that the proximal matrix $T$ in \eqref{equ:sADMM2} could be indefinite if $\alpha \in (0,1)$ though it provided no further discussions on theoretical properties.
Then Li et al. \cite{li2016majorized} proved the global convergence. He et al. \cite{he2017optimal} proposed a linearized version of ADMM with a positive-indefinite proximal term. They considered the case that matrix $S=0$ and $\alpha = 1$ in \eqref{equ:sADMM}, and generated the proximal matrix $T$ as
\[\label{T0}
T = \tau r I - \beta B^\top B ~~\mathrm{with}~~ r > \beta\|B^\top B\|,~~\tau \in (0.75,1).
\]
The proximal matrix $T$ is not necessarily positive semidefinite. A smaller value $\tau \in (0.75,1)$ can ensure the convergence and also give better numerical performance.

How to choose the proximal term is also one of the important research topics for ADMM. The popular proximal term is always chosen as a constant matrix. He et al. \cite{HLHY2002} extended the work to allow the parameters $\beta$, proximal terms $T$ and $S$ to be replaced by some bounded sequences of positive definite matrices $\{T_k\}$ and $\{S_k\}$. The resulting ADMM is a variable metric proximal ADMM, which is also closely related to the inexact ADMM \cite{eckstein1992douglas, chen1994proximal, HLHY2002, yuan2005improvement, eckstein2017approximate, eckstein2018relative}. The convergences of such methods have been studied in \cite{lotito2009class, banert2016fixing, Goncalves2018} but a better selection of the sequence $\{T_k\}$ has not been provided.

Quite recently, Gu and Yamashita \cite{Gu2019bfgs} proposed to construct a variable positive semi-definite sequence $\{T_k\}$ with $T_k = B_k - \nabla^2_{xx} \L_{\beta}(x,y,\la)$ when $f$ is quadratic. Note that $M = \nabla^2_{xx} \L_{\beta}(x,y,\la)$ is a constant matrix. They generated $B_k$ via the BFGS update with respect to $M$ at every iteration. Gu and Yamashita \cite{Gu2019optimal} further extended such a proximal ADMM for more general convex optimization problems with the proximal term generated by the Broyden family update. In these ADMMs, the proximal terms $T_k$ contain some second order information on the augmented Lagrangian function.
The papers \cite{Gu2019bfgs,Gu2019optimal} report some numerical results for LASSO and L1 regularized logistic regression. The results show that the algorithms can get a solution faster than the general indefinite proximal ADMM whose proximal term is fixed. Another interesting numerical result in \cite{Gu2019bfgs,Gu2019optimal} is that a variable indefinite sequence via the BFGS update also shows a good performance.

Inspired by the variable metric semi-proximal ADMM \cite{Gu2019bfgs,Gu2019optimal} and the indefinite proximal ADMM \cite{he2017optimal}, it is worth considering ADMM with a sequence of indefinite proximal matrices.
We call the resulting ADMM a variable metric indefinite proximal ADMM ({\color{blue}VMIP-ADMM}).
Throughout our discussion, we always choose the stepsize $\alpha$ in \eqref{equ:sADMM3} be 1 as that in \cite{he2017optimal}, which is good enough for such methods in practice and simple for the convergence analysis.

We now introduce the whole update scheme of the VMIP-ADMM:
\begin{subnumcases}{\label{equ:ADMMB}}
\label{equ:ADB1} x^{k+1}=\arg\min_{x}\ \L_{\beta}(x,y^k,\la^k) + \frac12 \|x - x^{k}\|_{S}^2,\\
\label{equ:ADB2} y^{k+1}=\arg\min_{y}\ \L_{\beta}(x^{k+1},y,\la^{k}) + \frac12 \|y - y^{k}\|_{T_k}^2,\\
\label{equ:ADB3} \la^{k+1}=\la^{k}-\beta(A x^{k+1} + By^{k+1} - b),
\end{subnumcases}
where $S$ is a fixed positive semi-definite and $T_k$ is possibly indefinite.
Note that the VMIP-ADMM can unify the several existing ADMMs.
\begin{itemize}
  \item  Let $S = 0$, $T_k \equiv 0$, VMIP-ADMM reduces to the classical ADMM;
  \item  Let $S$ and $T_k\equiv T$ be positive semidefinite matrices, VMIP-ADMM turns to be the semi-proximal ADMM \eqref{equ:sADMM};
  \item  Let $\{T_k\}$ be a positive semidefinite sequence, that is, $T_k\succeq 0$ for all $k$. VMIP-ADMM becomes the variable semi-proximal ADMM;
  \item  Let $S=0$, $T_k\equiv T$ be a positive indefite matrix, VMIP-ADMM covers the indefinite-proximal ADMM proposed in \cite{he2017optimal}.
\end{itemize}

We present sufficient conditions on $\{T_k\}$ for the global convergence of VMIP-ADMM.
The proof is followed by the analysis technique in Gu et al. \cite{Gu2019indefinite}, which separated the constant indefinite term ``$T$'' into two semidefinite parts as $T = T_+ - T_- $.
Moreover, we provide a construction of the indefinite term $T_k$ via the BFGS update. We extend a useful theorem in \cite{Gu2019bfgs} for a special case when $y$-subproblems \eqref{equ:ADB2} are unconstrained quadratic programming problems. We construct the $T_k$ with $T_k= B_k - M$, where $M$ is the Hessian matrix of the augmented Lagrangian function \eqref{augL} and $B_k$ is generated by the BFGS update with respect to $\tau M$, $\tau<1$.
We also show that this construction of $T_k$ satisfies the above conditions for the global convergence property when $\tau \in (0.75,1)$.

The remaining parts of the paper are organized as follows. We first give notations and some preliminaries that will be useful for subsequent analysis in Section \ref{vmip}. Then we present sufficient conditions on the proximal matrices $\{T_k\}$ for the global convergence. In Section \ref{sec:bfgs}, we discuss the choices of proximal matrix $T_k$ that guarantees the global convergence. We also show how to determine the value of $\tau$. Some conclusions and future works are given in Section \ref{sec:concl}.

\section{Global convergence of the variable metric indefinite proximal ADMM \label{vmip}}
In this section, we show the global convergence of the variable metric indefinite proximal ADMM \eqref{equ:ADMMB} (VMIP-ADMM) for problem \eqref{cp}.
To this end, we first present optimality conditions of problem \eqref{cp} and some useful properties which will be frequently used in our analysis. Then we give sufficient conditions on $\{T_k\}$ under which VMIP-ADMM converges globally.

\subsection{Optimality conditions for problem (\ref{cp})}
Let $\Omega = \R^{n} \times \R^{n} \times \R^m.$ The KKT conditions of problem \eqref{cp} are written as:
\begin{subnumcases}{\label{kkt}}
 \label{kktx} \xi_x^* - A^\top \la^* = 0,\\
 \label{kkty} \xi_y^* - B^\top \la^* = 0,\\
 \label{kktla} A x^* + B y^* - b =0,\\
 \label{kktxi} \xi_x^* \in \partial f(x^*), \;\; \xi_y^* \in \partial g(y^*).
\end{subnumcases}
Let $\Omega^*$ be a set of $(x^*, y^*, \la^*)$ satisfying the KKT conditions \eqref{kkt}.

Throughout this paper, we make the following assumption.
\begin{assumption}\label{ass:opt}
The set $\Omega^*$ of KKT points is non-empty.
\end{assumption}

The optimality conditions of subproblems \eqref{equ:ADB1} and \eqref{equ:ADB2} can be obtained respectively that
\begin{equation*}
 (x - x^{k+1})^\top \left(\xi_x^{k+1} - A^\top \la^k + \beta A^\top (Ax^{k+1} + By^k - b) + S (x^{k+1} - x^k) \right)\geq 0, \;\; \forall x \in \R^{n},
\end{equation*}
and
\begin{equation*}
 (y - y^{k+1})^\top \left(\xi_y^{k+1} - B^\top \la^k + \beta B^\top (Ax^{k+1} + By^{k+1} - b) + T_k (y^{k+1} - y^k) \right) \geq 0, \;\; \forall y\in \R^{n},
\end{equation*}
where $\xi_x^{k+1} \in \partial f(x^{k+1})$ and $\xi_y^{k+1} \in \partial g(y^{k+1})$.

Since $\lambda^{k+1} = \la^k - \beta(Ax^{k+1} + By^{k+1} - b)$ from \eqref{equ:ADB3}, we have
\begin{equation*}
- A^\top \la^k + \beta A^\top (Ax^{k+1} - b) = -A^\top \la^{k+1} - \beta A^\top B y^{k+1}
\end{equation*}
and
\begin{equation*}
- B^\top \la^k + \beta B^\top (Ax^{k+1} + By^{k+1} - b) = - B^\top \la^{k+1} .
\end{equation*}
Then the above optimality conditions can be written as
\[\label{x-opt-s}
 (x - x^{k+1})^\top \left(\xi_x^{k+1} - A^\top \la^{k+1} +  \beta A^\top B(y^k - y^{k+1}) + S (x^{k+1} - x^k) \right) \geq 0, \;\; \forall x\in \R^{n},
\]
and
\[\label{y-opt-s}
 (y - y^{k+1})^\top \left(\xi_y^{k+1} - B^\top \la^{k+1} + T_k (y^{k+1} - y^k) \right) \geq 0, \;\; \forall y\in \R^{n}.
\]

\subsection{Notations and Conditions on \texorpdfstring{$\{T_k\}$}{TEXT}}

We use the following notations throughout this paper:
\begin{equation*}
u = \left(\begin{array}{c}x \\ y \end{array}\right),\;  w = \left(\begin{array}{c}x \\ y \\ \lambda \end{array}\right).
 \end{equation*}

Since the subdifferential mappings of the closed proper convex functions $f$ and $g$ are maximal monotone, there exist two positive semidefinite matrices $\Sigma_f$ and $\Sigma_g$ such that for all $x, \hat{x} \in \R^{n}$, $\xi_x \in \partial f(x)$, and $\hat{\xi}_x \in \partial f(\hat{x})$,
\[\label{x-mono}
( x - \hat{x} )^\top (\xi_x - \hat{\xi}_x )\geq \|x - \hat{x}\|^2_{\Sigma_f},
\]
and for all $y, \hat{y} \in \R^{n}$, $\xi_y \in \partial g(y)$, and $\hat{\xi}_y \in \partial g(\hat{y})$,
\[\label{y-mono}
( y - \hat{y} )^\top (\xi_y - \hat{\xi}_y )\geq \|y - \hat{y}\|^2_{\Sigma_g}.
\]
Let $\Sigma \in \R^{2n \times 2n}$ denote
\begin{equation*}
\Sigma = \left(\begin{array}{c c} \Sigma_f & 0 \\ 0 & \Sigma_g\end{array}\right).
\end{equation*}




We first give the conditions for $S$ and the indefinite proximal sequence $\{T_k\}$ to guarantee the global convergence.
\begin{condition}
\label{cond-main}
The matrix $S$ in \eqref{equ:ADB1} satisfies
\begin{itemize}
{ \item[(a)] ~~ $S + \frac1 2 \Sigma_f \succeq 0$;
  \item[(b)] ~~ $S + \Sigma_f + \beta A^\top A \succ 0$.
}
\end{itemize}
Moreover, for sequence $\{T_k\}$ generated in \eqref{equ:ADMMB}, there exist a non-negative sequence $\{\gamma_k\}$ and positive semidefinite sequences $\{T_+^k\}$ and $\{T_-\}$ such that
 \begin{itemize}
{   \item[(c)] ~~ $T_k = T_+^k - T_- $ for all $k$;
    \item[(d)] ~~ $T_k + \Sigma_g + \beta B^\top B \succ 0$ for all $k$;
  \item[(e)] ~~ $\frac{1}{1+\gamma_k} T_+^k \preceq T_+^{k+1} \preceq (1+\gamma_k) T_+^k, \; \; \forall k \geq 0,$  $\sum\limits_{k=0}^{\infty} \gamma_k < \infty$;
  \item[(f)] ~~ $T_{k+1} + \Sigma_g + \beta B^\top B \preceq (1+\gamma_k) (T_k + \Sigma_g + \beta B^\top B)$ for all $k$;
  \item[(g)] ~~ $\exists \; c\in (0, 0.5)$,  $T_k + \frac3 2 \Sigma_g - \frac{\gamma_{k-1}}{2} T_+^k - 2T_- + (\frac34 - \frac12 c)\beta B^\top B \succeq 0$  for all $k$.
}
\end{itemize}
\end{condition}
Condition (a) and (b) indicate that the proximal marrix $S$ is allowed to be a slight indefinite but no less than $-\frac12 \Sigma_f$.
Condition (c) decomposes the indefinite matrix $T_k$ to two positive semidefinite parts. Note that we require the second part $T_-$ be fixed.
This condition will play an important role in the main analysis.
Condition (d) allows $T_k$ to be indefinite.
Condition (e) and (f) are the boundness for positive semi-definite part $T_+^k$ and indefinite $T_k$, respectably. Condition (g) is a requirement for global convergence and also an important condition for us to discuss the range of the indefiniteness.

For simplicity, we further define the following matrices. For all $k$,
\[\label{def-matrix} P_k = \left(\begin{array}{c c}S & 0\\ 0 & T_k\end{array}\right), D_k = \left(\begin{array}{c c c} S & 0 & 0\\ 0 & T_k & 0\\ 0 & 0 & \frac1 \beta I\end{array}\right),  \mathrm{and} \; G_k = \left(\begin{array}{c c c} S+\Sigma_f & 0 & 0\\ 0 & T_k + \Sigma_g + \beta B^\top B & 0 \\ 0 & 0 & \frac{1}{\beta}I\end{array}\right),\]
where $S, T_k$ and $\beta$ are given in \eqref{equ:ADMMB}.

Moreover, we also define the following matrices
\begin{subequations}\label{def-t}
\begin{align}
  \label{def-t1}&\Gamma_k = T^k_+ + T_-, \;\; \forall k, \\
  \label{def-t2}&\Lambda_k = - \frac{\gamma_{k-1}}{2} T_+^k - 2T_- + \Sigma_g, \;\; \forall k,\\
  \label{def-t3}&\Delta_k = T_k + \frac3 2 \Sigma_g - \frac{\gamma_{k-1}}{2} T_+^k - 2T_- + (\frac34 - \frac12 c)\beta B^\top B, \;\; \forall k,
 \end{align}
\end{subequations}
where $\{\gamma_k\}$ is a sequence satisfying Condition \ref{cond-main}. Note that $\Gamma_k \succeq 0$ for all $k$.

\subsection{Technical lemmas for convergence analysis of the variable metric indefinite proximal ADMM \label{sub:property}}
In order to show that VMIP-ADMM converges to a solution of \eqref{cp} globally, we first give some properties for the sequence $\{w_k\} = \{(x^k, y^k, \la^k)\}$ generated by \eqref{equ:ADMMB}.
\begin{lemma}
\label{lemma:con1}
Let $\{w^k\}$ be generated by \eqref{equ:ADMMB}. Then, for given $w^* = (x^*, y^*, \la^*) \in \Omega^*$, we have
\begin{align}\label{equ:lemma:con1}
(w^{k+1} - w^*)^\top D_k (w^{k+1}-w^k) +\|u^{k+1} - u^*\|^2_{\Sigma} \leq \beta(Ax^{k+1} - Ax^*)^\top (By^{k+1} - By^k ).
\end{align}
\end{lemma}
\begin{proof}
By taking $x=x^*$ and $y=y^*$ in the optimality conditions \eqref{x-opt-s} and \eqref{y-opt-s}, respectively, we have
$$
 (x^{k+1} - x^*)^\top (\xi_x^{k+1} - A^\top \la^{k+1} +  \beta A^\top B(y^k - y^{k+1}) + S (x^{k+1} - x^k) ) \leq 0,
$$
and
$$
 (y^{k+1} - y^*)^\top (\xi_y^{k+1} - B^\top \la^{k+1} + T_k (y^{k+1} - y^k)) \leq 0,
$$
where $\xi_x^{k+1} \in \partial f(x^{k+1})$ and $\xi_y^{k+1} \in \partial g(y^{k+1})$.

The inequalities are further rearranged as
\[\label{ineq:ADB1}
 (x^{k+1} - x^*)^\top S (x^{k+1} - x^k) + (x^{k+1} - x^*)^\top(\xi_x^{k+1} - A^\top \la^{k+1}) \leq  \beta (Ax^{k+1} - Ax^*)^\top (By^{k+1} - By^k)
\]
and
\[\label{ineq:ADB2}
 (y^{k+1} - y^*)^\top T_k (y^{k+1} - y^k) + (y^{k+1} - y^*)^\top (\xi_y^{k+1} - B^\top \la^{k+1}) \leq 0.
\]

Moreover, from \eqref{x-mono}-\eqref{y-mono} with $x=x^{k+1}$, $y=y^{k+1}$, $\hat{x} = x^*$ and $\hat{y}=y^*$, we have
\[\label{f-mon}
(x^{k+1} - x^*)^\top (\xi_x^{k+1} - \xi_x^*) \geq \|x^{k+1} - x^*\|^2_{\Sigma_f},
\]
and
\[\label{g-mon}
( y^{k+1} - y^* )^\top (\xi_y^{k+1} - \xi_y^*) \geq \|y^{k+1} - y^*\|^2_{\Sigma_g},
\]
where $\xi_x^* \in \partial f(x^*)$ and $\xi_y^* \in \partial g(y^*)$ satisfy the KKT conditions \eqref{kktx} and \eqref{kkty}, respectively.
It then follows from \eqref{kktx} and \eqref{f-mon} that
\begin{align*}
(x^{k+1} - x^*)^\top (\xi_x^{k+1} - A^\top \la^{k+1}) & = (x^{k+1} - x^*)^\top (\xi_x^{k+1} - \xi_x^*) + (x^{k+1} - x^*)^\top (\xi_x^* - A^\top \la^{k+1}) \nn \\
& \geq \|x^{k+1} - x^*\|^2_{\Sigma_f} + (Ax^{k+1} - Ax^*)^\top (\la^* - \la^{k+1}).
\end{align*}
Combining this inequality and \eqref{ineq:ADB1}, we have
\begin{align}\label{lemma1:inex}
& {} (x^{k+1} - x^*)^\top S (x^{k+1} - x^k) + (Ax^{k+1} - Ax^*)^\top(\lambda^* - \lambda^{k+1}) + \|x^{k+1} - x^* \|_{\Sigma_f}^2 \nn \\
& {} \;\;\;\; \leq \beta (Ax^{k+1} - Ax^*)^\top(By^{k+1} - By^k).
\end{align}
In a similar way, we have from \eqref{kkty}, \eqref{ineq:ADB2} and \eqref{g-mon} that
\[\label{lemma1:iney}
 (y^{k+1} - y^*)^\top T_k (y^{k+1} - y^k) + (By^{k+1} - By^*)^\top(\lambda^* - \lambda^{k+1}) + \|y^{k+1} - y^* \|_{\Sigma_g}^2 \leq 0.
\]
Rearranging \eqref{equ:ADB3}, we have $ A x^{k+1} + By^{k+1} - b= \frac1 \beta \left(\la^{k}- \la^{k+1}\right)$. It then follows from \eqref{kktla} that
$$ Ax^{k+1} + By^{k+1} - Ax^* - By^* = \frac1\beta (\lambda^k - \lambda^{k+1}).$$
Adding \eqref{lemma1:inex} and \eqref{lemma1:iney}, and recalling the definition of $D_k$ and $\Sigma$, it holds that
\begin{align*}
& (w^{k+1} - w^*)^\top D_k (w^{k+1} - w^k) + \|u^{k+1} - u^* \|_{\Sigma}^2 \nn \\
& = (x^{k+1} - x^*)^\top S (x^{k+1} - x^k) + (y^{k+1} - y^*)^\top T_k (y^{k+1} - y^k) \nn \\
& \;\;\;\; + \frac1\beta (\lambda^{k+1} - \lambda^k)^\top( \lambda^{k+1} - \lambda^*) + \|u^{k+1} - u^* \|_{\Sigma}^2 \nn \\
& = (x^{k+1} - x^*)^\top S (x^{k+1} - x^k) + (y^{k+1} - y^*)^\top T_k (y^{k+1} - y^k) \nn \\
& \;\;\;\;+ (Ax^{k+1} + By^{k+1} - Ax^*- By^*)^\top(\lambda^* - \lambda^{k+1}) + \|u^{k+1} - u^* \|_{\Sigma}^2 \nn \\
& \leq \beta (Ax^{k+1} - Ax^*)^\top(By^{k+1} - By^k). \qedhere
\end{align*}
\end{proof}

The inequality \eqref{equ:lemma:con1} in Lemma \ref{lemma:con1} is further rearranged as follows.
\begin{lemma}\label{lemma:con11}
Let $\{w^k\}$ be generated by \eqref{equ:ADMMB}. Then, for given $w^* = (x^*, y^*, \la^*) \in \Omega^*$, we have
\begin{align}\label{ineq:lemma:con11}
& 2(w^{k+1} - w^*)^\top D_k (w^{k+1}-w^k)   \nonumber \\
& \leq 2(By^{k+1} - By^k)^\top (\lambda^k - \lambda^{k+1}) - 2\beta (By^{k+1} - By^*)^\top (By^{k+1} - By^k ) - 2\|u^{k+1} - u^*\|^2_{\Sigma}.  \qedhere
\end{align}
\end{lemma}
\begin{proof}
Noting that $Ax^* + By^* -b=0$, the twice of the right hand of \eqref{equ:lemma:con1} is written as
\begin{align*}
& 2\beta(Ax^{k+1} - Ax^*)^\top (By^{k+1} - By^k )  \nonumber \\
& = 2 \beta (A x^{k+1} + By^* - b  + By^{k+1} -By^{k+1})^ \top (By^{k+1} - By^k )  \nonumber \\
& = 2\beta(Ax^{k+1} + By^{k+1} - b)^\top (By^{k+1} - By^k ) - 2\beta (By^{k+1} - By^*)^\top (By^{k+1} - By^k )  \nonumber \\
& = 2(By^{k+1} - By^k)^\top (\lambda^k - \lambda^{k+1}) - 2\beta (By^{k+1} - By^*)^\top (By^{k+1} - By^k ),
\end{align*}
where the last equality follows from \eqref{equ:ADB3}.
Then the assertion is directly obtained from \eqref{equ:lemma:con1}.
\end{proof}

Next we give a simple but important lemma.
\begin{lemma}\label{tech}
For vectors $a, b\in \R^n$, and symmetric positive semidefinite matrices $M_1, M_2 \in \R^{n\times n}$, we have that
\[\label{ineq:tech}
a^\top M_1 b - a^\top M_2 b \leq \frac 12 a^\top (M_1+ M_2) a + \frac 12 b^\top (M_1+ M_2) b.
\]
\end{lemma}
\begin{proof}
For a positive semidefinite matrix $M_1$, we have
$$0 \leq \frac 12 \|a - b\|_{M_1}^2 = \frac 12 a^\top M_1 a + \frac 12 b^\top M_1 b - a^\top M_1 b,$$
which implies
\[\label{ineq:tech:1}
a^\top M_1 b \leq \frac 12 a^\top M_1 a + \frac 12 b^\top M_1 b.
\]
In a similar way for $M_2$, we have
\[\label{ineq:tech:2}
- a^\top M_2 b \leq \frac 12 a^\top M_2 a + \frac 12 b^\top M_2 b.\]
The assertion immediately follows by adding \eqref{ineq:tech:1} and \eqref{ineq:tech:2}.
\end{proof}

In order to bound $(w^{k+1} - w^*)^\top D_k (w^{k+1}-w^k)$ further, we now give two technical lemmas to estimate upper-bounds for the crossing term $(By^{k+1} - By^k)^\top (\lambda^k - \lambda^{k+1})$ in \eqref{ineq:lemma:con11}.
\begin{lemma}
\label{lemma:con2}
Let $\{w^k\}$ be generated by the scheme \eqref{equ:ADMMB}. Suppose that the proximal sequence $\{T_k\}$ satisfies Condition \ref{cond-main}. Then it holds that
\begin{align}\label{ineq:lemma:con2}
(By^{k+1} - By^k)^\top (\lambda^k - \lambda^{k+1}) \leq \frac12 \|y^{k-1} - y^k\|_{\Gamma_{k-1}}^2 - \frac12 \|y^{k+1} - y^k\|_{\Gamma_{k}}^2 - \|y^{k+1} - y^k\|_{\Lambda_k}^2,
\end{align}
where $\Gamma_{k}$ and $\Lambda_k$ are defined in \eqref{def-t}.
\end{lemma}
\begin{proof}
From the optimality condition \eqref{y-opt-s} for $y^{k+1}$, we can easily derive the optimality condition for $y^k$ as
\[\label{lemma:con2:ineq2}
(y - y^{k})^\top (\xi_y^{k} - B^\top \lambda^{k} + T_{k-1} (y^{k}-y^{k-1}))\geq 0, \;\; \forall y \in \R^{n}.
\]
Choosing $y=y^k$ in \eqref{y-opt-s}, we have
\begin{align}\label{lemma:con2:ineq12}
0 & \leq (y^k - y^{k+1})^\top (\xi_y^{k+1} - B^\top \lambda^{k+1} + T_k (y^{k+1}-y^k)) \nn \\
& = (y^{k+1} - y^k)^\top (-\xi_y^{k+1} + B^\top \lambda^{k+1} - T_k (y^{k+1}-y^k)).
\end{align}
Moreover, letting $y=y^{k+1}$ in \eqref{lemma:con2:ineq2}, we have
\[\label{lemma:con2:ineq22}
0 \leq (y^{k+1} - y^{k})^\top (\xi_y^{k} - B^\top \lambda^{k} + T_{k-1} (y^{k}-y^{k-1})).
\]
Summing inequalities \eqref{lemma:con2:ineq12} and \eqref{lemma:con2:ineq22}, we obtain that
\begin{align*}
0 \leq & -(y^{k+1}- y^{k})^\top (\xi_y^{k+1} - \xi_y^{k}) + (By^{k+1} - By^k)^\top (\lambda^{k+1} - \lambda^k)  \nn \\
& + (y^{k+1}-y^k)^\top T_{k-1}(y^{k}-y^{k-1}) - \|y^{k+1} - y^k\|^2_{T_k}.
\end{align*}
It then follows from \eqref{y-mono} that
$$
0 \leq - \|y^{k+1} - y^k\|^2_{\Sigma_g} + (By^{k+1} - By^k)^\top (\lambda^{k+1} - \lambda^k) + (y^{k+1}-y^k)^\top T_{k-1}(y^{k}-y^{k-1}) - \|y^{k+1} - y^k\|^2_{T_k}, $$
which is equivalent to
\[\label{lemma:con2:ineq4}
(By^{k+1} - By^k)^\top (\lambda^k - \lambda^{k+1}) \leq - \|y^{k+1} - y^k\|^2_{T_k} + (y^{k+1}-y^k)^\top T_{k-1}(y^{k}-y^{k-1}) - \|y^{k+1} - y^k\|^2_{\Sigma_g}. \]

Recall that $T_{k-1} = T_+^{k-1} - T_- $ from (c) in Condition \ref{cond-main} and $T_+^{k-1}, T_- \succeq 0$. Then we have
\begin{align}\label{lemma:con2:ineq3}
(y^{k+1}-y^k)^\top T_{k-1}(y^{k}-y^{k-1}) & {} = (y^{k+1}-y^k)^\top T_+^{k-1}(y^{k}-y^{k-1}) - (y^{k+1}-y^k)^\top T_-(y^{k}-y^{k-1}) \nn \\
&{} \leq  \frac12 \|y^{k+1} - y^k\|_{T_+^{k-1} + T_-}^2  + \frac12 \|y^{k-1} - y^k\|_{T_+^{k-1} + T_-}^2,
\end{align}
where the inequality follows from \eqref{ineq:tech} with $a = (y^{k+1}-y^k)$, $b = (y^{k}-y^{k-1})$, $M_1 = T_+^{k-1}$ and $M_2 = T_-$.

We then have from \eqref{lemma:con2:ineq4} that
\begin{align*}
&{} (By^{k+1} - By^k)^\top (\lambda^k - \lambda^{k+1}) \nonumber \\
& \leq - \|y^{k+1} - y^k\|^2_{T_k} + (y^{k+1}-y^k)^\top T_{k-1}(y^{k}-y^{k-1}) - \|y^{k+1} - y^k\|^2_{\Sigma_g}  \nonumber \\
& \leq  - \|y^{k+1} - y^k\|^2_{T_+^k - T_-} + \frac12 \|y^{k+1} - y^k\|_{T_+^{k-1} + T_-}^2  + \frac12 \|y^{k-1} - y^k\|_{T_+^{k-1} + T_-}^2  - \|y^{k+1} - y^k\|^2_{\Sigma_g}  \nonumber \\
& \leq  - \|y^{k+1} - y^k\|^2_{T_+^k - T_-} + \frac12 \|y^{k+1} - y^k\|_{(1+\gamma_{k-1}) T_+^k + T_-}^2  + \frac12 \|y^{k-1} - y^k\|_{T_+^{k-1} + T_-}^2 - \|y^{k+1} - y^k\|^2_{\Sigma_g}  \nonumber \\
& = \frac12 \|y^{k-1} - y^k\|_{T_+^{k-1} + T_-}^2 - \frac12 \|y^{k+1} - y^k\|_{T_+^{k} + T_-}^2 - \|y^{k+1} - y^k\|_{-\frac{\gamma_{k-1}}{2}T_+^k -2T_- + \Sigma_g}^2 \nn \\
& = \frac12 \|y^{k-1} - y^k\|_{\Gamma_{k-1}}^2 - \frac12 \|y^{k+1} - y^k\|_{\Gamma_{k}}^2 - \|y^{k+1} - y^k\|_{\Lambda_k}^2,
\end{align*}
where the second inequality follows from $T_k = T_+^{k} - T_- $ and \eqref{lemma:con2:ineq3}, the third inequality follows from Condition \ref{cond-main} (d), and the last equality is from the definitions \eqref{def-t1} and \eqref{def-t2}.
Then it shows the assertion \eqref{ineq:lemma:con2}.
\end{proof}

Besides Lemma \ref{lemma:con2}, we can derive another estimation for $(By^{k+1} - By^k)^\top (\lambda^k - \lambda^{k+1})$, whose proof is similar to that in \cite[Lemma 4.4]{he2017optimal}.
\begin{lemma}
\label{lemma:con22}
Let $\{w^k\}$ be generated by the scheme \eqref{equ:ADMMB}. Then, for any $c\in (0, 0.5)$, it holds that
\[\label{ineq:lemma:con22}
(By^{k+1} - By^k)^\top (\lambda^k - \lambda^{k+1}) \leq \left(\frac14 + \frac12 c\right) \beta\|By^{k+1} - By^k\|^2 + (1-c)\frac1 \beta \|\lambda^{k+1} - \lambda^k\|^2.
\]
\end{lemma}
\begin{proof}
See \cite[Lemma 4.4]{he2017optimal}.
\end{proof}

\medskip
Based on the above two lemmas for $(By^{k+1} - By^k)^\top (\lambda^k - \lambda^{k+1})$, we can further bound $(w^{k+1} - w^*)^\top D_k (w^{k+1}-w^k)$ in \eqref{ineq:lemma:con11} of Lemma \ref{lemma:con11}.
\begin{lemma}
\label{lemma:con12}
Let $\{w^k\}$ be generated by \eqref{equ:ADMMB}. Suppose that the proximal sequence $\{T_k\}$ satisfies Condition \ref{cond-main}. Then, for given $w^* = (x^*, y^*, \la^*) \in \Omega^*$, we have
\begin{align}\label{ineq:lemma:con12}
&{} 2(w^{k+1} - w^*)^\top D_k (w^{k+1}-w^k)   \nonumber \\
&{} \leq \frac12 \|y^{k-1} - y^k\|_{\Gamma_{k-1}}^2 - \frac12 \|y^{k+1} - y^k\|_{\Gamma_{k}}^2 - \|y^{k+1} - y^k\|_{\Lambda_k}^2 + \left(\frac14 + \frac12 c\right) \beta\|By^{k+1} - By^k\|^2 \nonumber \\
&{} \;\;\;\; +\frac{(1-c)} {\beta} \|\lambda^{k+1} - \lambda^k\|^2 - 2\beta (By^{k+1} - By^*)^\top (By^{k+1} - By^k ) - 2\|u^{k+1} - u^*\|^2_{\Sigma}.   \qedhere
\end{align}
\end{lemma}
\begin{proof}
The term $2(By^{k+1} - By^k)^\top (\lambda^k - \lambda^{k+1})$ in inequality \eqref{ineq:lemma:con11} can be bounded by the above lemmas \eqref{ineq:lemma:con2} and \eqref{ineq:lemma:con22}, and then the assertion is obtained.
\end{proof}

\subsection{Global Convergence of the variable metric indefinite proximal ADMM\label{sub:con}}
In this subsection we show the global convergence based on the results in the previous subsection and Condition \ref{cond-main}. Firstly, we obtain the following contractive result, which will play a key role in proving the convergence of \eqref{equ:ADMMB}.
\begin{lemma}
\label{lemma:con3}
Let $w^* = (x^*, y^*, \la^*) \in \Omega^*$, and let $\{w^k\}$ be generated by the scheme \eqref{equ:ADMMB}. Suppose that the proximal sequence $\{T_k\}$ satisfies Condition \ref{cond-main}. Then we have
\begin{align}\label{equ:lemma:con3}
&\|w^k - w^*\|_{G_k}^2  + \frac12\|y^{k-1} - y^k\|_{\Gamma_{k-1}}^2 - \left( \|w^{k+1} - w^*\|_{G_k}^2 + \frac12 \|y^{k+1} - y^k\|_{\Gamma_{k}}^2 \right)  \nonumber \\
&{}\geq \underbrace{\|x^{k+1} - x^k\|_{S+\frac12 \Sigma_f}^2  + \|y^{k+1} - y^k\|_{\Delta_k}^2 + \frac c \beta  \|\lambda^{k+1} - \lambda^k\|^2}_{\mathrm{Term1}},
\end{align}
where $\Gamma_{k}$ and $\Delta_k$ are given in \eqref{def-t}.
\end{lemma}
\begin{proof}
By the identity $ \|a+b\|^2 = \|a\|^2 - \|b\|^2 + 2(a+b)^\top b$, we get
\begin{align*}
\|By^{k+1} - By^*\|^2 & = \|By^k - By^* + By^{k+1} - By^k\|^2 \nn \\
& = \|By^k - By^*\|^2 - \|By^{k+1} - By^k\|^2 + 2(By^{k+1} - By^*)^\top ( By^{k+1} - By^k).
\end{align*}
Moreover,
\begin{align*}
\|w^{k+1} - w^*\|_{D_k}^2 & = \|w^k - w^* + w^{k+1} - w^k\|_{D_k}^2 \nn \\
& = \|w^k - w^*\|_{D_k}^2 - \|w^{k+1} - w^k\|_{D_k}^2 + 2(w^{k+1} - w^*)^\top D_k ( w^{k+1} - w^k).
\end{align*}
Then we have
\begin{align}\label{lemma:con3:ineq1}
&{} \|w^{k+1} - w^*\|_{D_k}^2  +  \beta \|By^{k+1} - By^*\|^2  \nonumber \\
&{} = \|w^{k} - w^*\|_{D_k}^2  +  \beta \|By^{k} - By^*\|^2 - \left( \|w^{k+1} - w^k\|_{D_k}^2  +  \beta \|By^{k+1} - By^k\|^2 \right) \nonumber \\
&{} \;\;\;\; + 2(w^{k+1} - w^*)^\top D_k (w^{k+1}-w^k) + 2\beta (By^{k+1} - By^*)^\top (By^{k+1} - By^k).  
\end{align}
Since the term $ 2(w^{k+1} - w^*)^\top D_k (w^{k+1}-w^k)$ in equality \eqref{lemma:con3:ineq1} can be bounded by \eqref{ineq:lemma:con12} in Lemma \ref{lemma:con12}, we can rearrange \eqref{lemma:con3:ineq1} as
\begin{align}\label{lemma:con3:ineq3}
&{} \|w^{k+1} - w^*\|_{D_k}^2  +  \beta \|By^{k+1} - By^*\|^2  \nonumber \\
&{} \leq \|w^{k} - w^*\|_{D_k}^2  +  \beta \|By^{k} - By^*\|^2 - \|w^{k+1} - w^k\|_{D_k}^2  -  \beta \|By^{k+1} - By^k\|^2  \nonumber \\
&{} \;\;\;\; + \frac12 \|y^{k-1} - y^k\|_{\Gamma_{k-1}}^2 - \frac12 \|y^{k+1} - y^k\|_{\Gamma_{k}}^2 - \|y^{k+1} - y^k\|_{\Lambda_k}^2   \nonumber \\
&{} \;\;\;\; + \left(\frac14 + \frac12 c\right) \beta\|By^{k+1} - By^k\|^2 + (1-c)\frac1 \beta \|\lambda^{k+1} - \lambda^k\|^2 - 2\|u^{k+1} - u^* \|_{\Sigma}^2  \nonumber \\
&{} = \|w^{k} - w^*\|_{D_k}^2  +  \beta \|By^{k} - By^*\|^2 & \nn \\
&{} \;\;\;\; - \|u^{k+1} - u^k\|_{P_k}^2  - \left(\frac34 - \frac12 c\right)\beta \|By^{k+1} - By^k\|^2 - \frac c \beta \|\lambda^{k+1} - \lambda^k\|^2 \nonumber \\
&{} \;\;\;\; + \frac12 \|y^{k-1} - y^k\|_{\Gamma_{k-1}}^2 - \frac12 \|y^{k+1} - y^k\|_{\Gamma_{k}}^2 - \|y^{k+1} - y^k\|_{\Lambda_k}^2 - 2\|u^{k+1} - u^* \|_{\Sigma}^2,
\end{align}
where the last equality follows from the definitions of $P_k$ and $D_k$ in \eqref{def-matrix}.
Rearranging \eqref{lemma:con3:ineq3} further, we have
\begin{align*}
&{} \|w^{k+1} - w^*\|_{D_k}^2  +  \beta \|By^{k+1} - By^*\|^2 + \frac12 \|y^{k+1} - y^k\|_{\Gamma_{k}}^2 \nn \\
&{} \leq \|w^{k} - w^*\|_{D_k}^2  +  \beta \|By^{k} - By^*\|^2 + \frac12 \|y^{k-1} - y^k\|_{\Gamma_{k-1}}^2   - 2\|u^{k+1} - u^* \|_{\Sigma}^2 \nn \\
&{} \;\;\;\; - \left(\|u^{k+1} - u^k\|_{P_k}^2 + \frac c \beta \|\lambda^{k+1} - \lambda^k\|^2 + \|y^{k+1} - y^k\|_{\Lambda_k + (\frac34 - \frac12 c)\beta B^\top B}^2\right),
\end{align*}
that is,
\begin{align}\label{lemma:con3:ineq4}
&{} \|w^{k} - w^*\|_{D_k}^2  +  \beta \|By^{k} - By^*\|^2 + \frac12 \|y^{k-1} - y^k\|_{\Gamma_{k-1}}^2 + \|u^k - u^*\|_{\Sigma}^2  \nn \\
&{} \;\;\;\; - \left(\|w^{k+1} - w^*\|_{D_k}^2  +  \beta \|By^{k+1} - By^*\|^2 + \frac12 \|y^{k+1} - y^k\|_{\Gamma_{k}}^2 + \|u^{k+1} - u^*\|_{\Sigma}^2 \right)  \nn \\
&{} \geq \|u^{k+1} - u^k\|_{P_k}^2 + \frac c \beta \|\lambda^{k+1} - \lambda^k\|^2 + \|y^{k+1} - y^k\|_{\Lambda_k + (\frac34 - \frac12 c)\beta B^\top B}^2 \nn \\
&{} \;\;\;\; + \|u^k - u^*\|_{\Sigma}^2 - \|u^{k+1} - u^*\|_{\Sigma}^2 + 2\|u^{k+1} - u^* \|_{\Sigma}^2.
\end{align}
From the definition of $G_k$ in \eqref{def-matrix}, inequality \eqref{lemma:con3:ineq4} can be written as
\begin{align*}
&\|w^k - w^*\|_{G_k}^2  + \frac12\|y^{k-1} - y^k\|_{\Gamma_{k-1}}^2 - \left( \|w^{k+1} - w^*\|_{G_k}^2 + \frac12 \|y^{k+1} - y^k\|_{\Gamma_{k}}^2 \right)  \nonumber \\
&{} \geq \|u^{k+1} - u^k\|_{P_k}^2 + \frac c \beta \|\lambda^{k+1} - \lambda^k\|^2 + \|y^{k+1} - y^k\|_{\Lambda_k + (\frac34 - \frac12 c)\beta B^\top B}^2 + \|u^k - u^*\|_{\Sigma}^2 + \|u^{k+1} - u^* \|_{\Sigma}^2  \nonumber \\
&{} \geq \|u^{k+1} - u^k\|_{P_k}^2 + \frac c \beta \|\lambda^{k+1} - \lambda^k\|^2 + \|y^{k+1} - y^k\|_{\Lambda_k + (\frac34 - \frac12 c)\beta B^\top B}^2 + \frac12 \|u^{k+1} - u^k\|_{\Sigma}^2  \nonumber \\
&{} = \|x^{k+1} - x^k\|_{S+\frac12 \Sigma_f}^2  + \|y^{k+1} - y^k\|_{T_k + \Lambda_k + (\frac34 - \frac12 c)\beta B^\top B + \frac12 \Sigma_g}^2 + \frac c \beta  \|\lambda^{k+1} - \lambda^k\|^2,
\end{align*}
where the second inequality follows from the well-known inequality $\|a\|_M^2 + \|b\|_M^2 \geq \frac 12 \|a-b\|_M^2$ with $M = \Sigma$, $a = u^k - u^*$ and $b = u^{k+1} - u^*$.

From the definitions \eqref{def-t2} and \eqref{def-t3}, we have that
$$ \Delta_k = T_k + \frac12 \Sigma_g + \Lambda_k + (\frac34 - \frac12 c)\beta B^\top B.$$
Thus the proof is completed.
\end{proof}
Condition \ref{cond-main} (a) implies $\|x^{k+1} - x^k\|_{S+\frac12 \Sigma_f}^2 \geq 0$ for all $k$. Moreover,
Condition \ref{cond-main} (g) implies $\|y^{k+1} - y^k\|_{\Delta_k}^2 \geq 0$ for all $k$. Therefore, Term1 in \eqref{equ:lemma:con3} is always nonnegative, which indicates the contraction of the sequence $\{w_k\}$.

\medskip
It follows from the definition of $\{G_k\}$ and Condition \ref{cond-main} (a), (c) and (e) that $0 \preceq G_{k+1} \preceq (1+\gamma_k) G_k$ for all $k$.
We define two constants $C_s$ and $C_p$ as follows:
\begin{equation*}
C_s\colon = \sum_{k=0}^{\infty} \gamma_k  \; \; \mathrm{and} \;\;  C_p\colon = \prod_{k=0}^{\infty} (1+\gamma_k).
\end{equation*}
From the assumption $\sum_{0}^{\infty} \gamma_k < \infty$ and $\gamma_k \geq 0$, we have $0 \leq C_s < \infty$ and $1 \leq C_p < \infty$. Moreover, we can easily get
$$ 0 \preceq G_{k} \preceq C_p G_0, \; \; \forall k \geq 0,$$
which means that the sequences $\{G_k\}$ is bounded.

\medskip
Now we give the main convergent theorem of this subsection.
\begin{theorem}
\label{theo:conv}
Let $w^* = (x^*, y^*, \la^*) \in \Omega^*$, and let $\{w^k\}$ be a sequence generated by \eqref{equ:ADMMB}. Suppose that $\{T_k\}$ is a sequence satisfying Condition \ref{cond-main}. Then the sequence $\{w^k\}$ converges to a point $w^* \in \Omega^*$.
\end{theorem}
\begin{proof}
First we show that the sequence $\{w^k\}$ is bounded. Since $0 \preceq G_{k+1} \preceq (1+\gamma_k) G_k$,
we have
\[\label{theo:2:2}
\|w^{k+1} - w^*\|_{G_{k+1}}^2 \leq (1+\gamma_k)\|w^{k+1} - w^*\|_{G_k}^2.
\]
Combining the inequality \eqref{theo:2:2} with \eqref{equ:lemma:con3} in Lemma \ref{lemma:con3}, we have
\begin{align}\label{theo:conv:1}
&{} \|w^{k+1} - w^*\|_{G_{k+1}}^2 + \frac12 \|y^{k+1} - y^k\|_{\Gamma_k}^2 \nn \\
&{} \overset{(\ref{theo:2:2})}{\leq}(1+\gamma_k)\left(\|w^{k+1} - w^*\|_{G_{k}}^2 + \frac12 \|y^{k+1} - y^k\|_{\Gamma_k}^2\right) \nn \\
&{} \overset{(\ref{equ:lemma:con3})}{\leq} (1+\gamma_k)\left(\|w^k - w^*\|_{G_k}^2  + \frac12\|y^{k-1} - y^k\|_{\Gamma_{k-1}}^2\right) - (1+\gamma_k) \mathrm{Term1} \nn \\
&{}\;\; \leq \;\; (1+\gamma_k)\left(\|w^k - w^*\|_{G_k}^2  + \frac12\|y^{k-1} - y^k\|_{\Gamma_{k-1}}^2\right) - \mathrm{Term1}.
\end{align}
It then follows that for all $k$,
\begin{align}\label{theo:conv:2}
 \|w^{k+1} - w^*\|_{G_{k+1}}^2 + \frac12 \|y^{k+1} - y^k\|_{\Gamma_k}^2  &{}\leq  \left(\prod_{i=0}^{k} (1+\gamma_i)\right) \left(\|w^0 - w^*\|_{G_0}^2 + \frac12\|y^{0} - y^1\|_{\Gamma_0}^2 \right) \nn \\
&{} \leq C_p \left(\|w^0 - w^*\|_{G_0}^2 + \frac12\|y^{0} - y^1\|_{\Gamma_0}^2 \right).
\end{align}
Note that
\[\label{theo:conv:3}
\|w^{k+1} - w^*\|_{G_{k+1}}^2 = \|x^{k+1} - x^*\|_{S+\Sigma_f}^2 + \|y^{k+1} - y^*\|_{T_k + \Sigma_g + \beta B^\top B}^2 + \frac 1 \beta \|\lambda^{k+1} - \lambda^*\|^2,\]
$T_k + \Sigma_g + \beta B^\top B$ is positive definite from Condition \ref{cond-main} (d), and $C_p \left(\|w^0 - w^*\|_{G_0}^2 + \frac12\|y^{0} - y^1\|_{\Gamma_0}^2 \right)$ is a constant. It then follows from \eqref{theo:conv:2} that $\{y^k\}$ and $\{\la^k\}$ are bounded. We now show that $\{x^k\}$ is also bounded.

From \eqref{theo:conv:1} and \eqref{theo:conv:2}, we have
\begin{align*}
\mathrm{Term1} &{} = \|x^{k+1} - x^k\|_{S+\frac12 \Sigma_f}^2  + \|y^{k+1} - y^k\|_{\Delta_k}^2 + \frac c \beta  \|\lambda^{k+1} - \lambda^k\|^2 \nn \\
&{}\leq \|w^k - w^*\|_{G_k}^2 - \|w^{k+1} - w^*\|_{G_{k+1}}^2 + \frac12\|y^{k-1} - y^k\|_{\Gamma_{k-1}}^2- \frac12 \|y^{k+1} - y^k\|_{\Gamma_k}^2 \nn \\
&{} \;\;\;\; + \gamma_k \left(\|w^k - w^*\|_{G_k}^2  + \frac12\|y^{k-1} - y^k\|_{\Gamma_{k-1}}^2\right) \nn \\
&{}\leq \|w^k - w^*\|_{G_k}^2 - \|w^{k+1} - w^*\|_{G_{k+1}}^2 + \frac12\|y^{k-1} - y^k\|_{\Gamma_{k-1}}^2 - \frac12 \|y^{k+1} - y^k\|_{\Gamma_k}^2 \nn \\
&{} \;\;\;\; + C_p \left(\|w^0 - w^*\|_{G_0}^2 + \frac12\|y^{0} - y^1\|_{\Gamma_0}^2 \right).
\end{align*}
Summing up the inequalities, we obtain
\begin{align*}
&{}\sum_{k=1}^{\infty} \left( \|x^{k+1} - x^k\|_{S+\frac12 \Sigma_f}^2  + \|y^{k+1} - y^k\|_{\Delta_k}^2 + \frac c \beta  \|\lambda^{k+1} - \lambda^k\|^2\right) \nn \\
&{} \leq{} \|w^0 - w^*\|_{G_0}^2 + \frac12\|y^{0} - y^1\|_{\Gamma_0}^2 + \left(\sum_{k=0}^{\infty}\gamma_k\right) C_p \left(\|w^0 - w^*\|_{G_0}^2 + \frac12\|y^{0} - y^1\|_{\Gamma_0}^2 \right) \nn \\
&{} \leq{}  (1+ C_s C_p) \left(\|w^0 - w^*\|_{G_0}^2 + \frac12\|y^{0} - y^1\|_{\Gamma_0}^2 \right).
\end{align*}
Since $(1+ C_s C_p) \left(\|w^0 - w^*\|_{G_0}^2 + \frac12\|y^{0} - y^1\|_{\Gamma_0}^2 \right)$ is a finite constant, we have
\begin{equation*}
\lim_{k\rightarrow \infty} \|x^{k+1} - x^k\|_{S+\frac12 \Sigma_f}^2  + \|y^{k+1} - y^k\|_{\Delta_k}^2 + \frac c \beta  \|\lambda^{k+1} - \lambda^k\|^2 = 0,
\end{equation*}
which indicates that
\[\label{theo:lim} \lim_{k\rightarrow \infty} \|\lambda^{k+1} - \lambda^k\| = \lim_{k\rightarrow \infty} \beta \|Ax^{k+1}+ By^{k+1} - b\| = 0. \]
Note that $Ax^* + By^*-b =0$, and
$$\|Ax^{k+1} - Ax^*\| = \|Ax^{k+1} + By^{k+1} - b - B(y^{k+1} - y^k)\| \leq \|Ax^{k+1}+ By^{k+1} - b\| + \|B(y^{k+1} - y^k)\|.$$
It then follows from \eqref{theo:lim} that $\|A(x^{k+1} - x^*)\|$ is bounded.
Moreover, inequalities \eqref{theo:conv:2} and \eqref{theo:conv:3} imply $\|x^{k+1} - x^*\|_{S+\Sigma_f}^2 $ is bounded. Therefore $\|x^{k+1} - x^*\|_{S+\Sigma_f+\beta A^\top A }^2 $ is abounded since
$$\|x^{k+1} - x^*\|_{S + \Sigma_f + \beta A^\top A}^2 = \|x^{k+1} - x^*\|_{S+\Sigma_f}^2 + \beta\|A(x^{k+1} - x^*)\|^2.$$
From the positive definiteness of $S + \Sigma_f + \beta A^\top A$ in Condition \ref{cond-main} (b), it shows that $\{x^k\}$ is also bounded.
Consequently, the sequence $\{w^k\}$ is bounded.

Next we should show that any cluster point of the sequence $\{w^k\}$ is an optimal solution of \eqref{cp} and the sequence $\{w^k\}$ has only one cluster point. This can be done in a way similar to the proof of that in \cite{Gu2019bfgs}.
\end{proof}

\section{VMIP-ADMM with the BFGS update}\label{sec:bfgs}
As shown in the recent researches \cite{Gu2019bfgs, Gu2019optimal}, a special variable metric proximal term via the BFGS update can get a solution faster on the iteration and CPU time than the proximal ADMM \cite{fazel2013hankel, he2017optimal} with a fixed proximal matrix $T$.
Moreover, in their experiments, a slightly indefinite variable also performs well without the theoretical analysis.
Note that this choice should have an assumption that the $y$-subproblems \eqref{equ:ADB2} should be unconstrained quadratic programming problem.
Based on the analysis above and the previous studies, we propose indefinite proximal terms $\{T_k\}$ updated by the BFGS update, and show that $\{T_k\}$ satisfies Condition \ref{cond-main}.

\subsection{Construction of the indefinite proximal matrix \texorpdfstring{$T_k$}{TEXT} via the BFGS update}
Inspired by the semidefinite proximal ADMM with the BFGS update \cite{Gu2019bfgs, Gu2019optimal}, we construct the indefinite matrix $T_k$ by the BFGS update.

We first explain the pure BFGS update for the following unconstrained quadratic optimization:
\begin{equation*}
 \min \,  \frac12 x^\top M x,
\end{equation*}
where $M \in \R^{n\times n}$ is a positive definite matrix. Let $s\in \R^n$ and $l = Ms$.
Note that $s^\top l >0$ when $s\neq 0$.
The BFGS update generates a sequence of approximate matrices $\{B_k\}$ of $M$, and its inverse $H_k =B_k^{-1}$.
For a given matrix $B_k$, the BFGS update generates $B_{k+1}^{\mathrm{BFGS}}$ and $H_{k+1}^{\mathrm{BFGS}}$ with $s$ and $l$ as follows
\[\label{bfgsb}
B^{\rm BFGS}_{k+1}=B_{k}+{\frac {{l}{l}^\top}{{l}^\top {s}}}-{\frac {B_{k}{s}{s}^\top B_{k}^\top}{{s}^\top B_{k}{s}}},
\]
\[\label{bfgsh}
H^{\rm BFGS}_{k+1}=\left(I-{\frac {sl^\top}{s^\top l}}\right) H_k \left(I-{\frac {ls^\top}{s^\top l}}\right)+{\frac {ss^\top}{s^\top l}}.
\]
Note that $B^{\rm BFGS}_{k+1}$ and $H^{\rm BFGS}_{k+1}$ are positive definite whenever $B_k, H_k \succ 0$ since $s^\top l > 0$. Note also that $H^{\rm BFGS}_{k+1} l = s = M^{-1} l$.

We now explain how to construct $T_k$ via the BFGS update.
Throughout this section we suppose that $g$ in the objective function \eqref{cp} is a convex quadratic function. Then $y$-subproblems \eqref{equ:ADB2} are unconstrained quadratic programming problems, and the Hessian matrix of the augmented Lagrangian function \eqref{augL} is a constant matrix given as 
\begin{equation*}
M \colon = \nabla^2_{yy} \L_{\beta}(x,y,\la) = \bar{M} + \beta B^\top B, 
\end{equation*}
where $\bar{M} \colon = \nabla^2_{yy} g(y)$.
Note that $M$ is always positive semidefinite since $\bar{M} \succeq 0$.

We consider a perturbed matrix $M^{\delta} \colon = M + \delta I \succ 0$ with a sufficiently small $\delta > 0$, and construct an approximate matrix $B_k$ of $M^{\delta}$ via the BFGS update \eqref{bfgsb}. Let $s_k = x^{k+1} - x^k$, where $\{x^k\}$ is a sequence generated by \eqref{equ:ADMMB}.
We propose that $\{B_k\}$ is generated as
\[\label{cond2}
B_{k+1}=B_{k}+ c_k \left(\frac {{\tilde{l}}_{k}{\tilde{l}}_{k}^\top}{{\tilde{l}}_{k}^\top {s}_{k}} -{\frac {B_{k}{s}_{k}{s}_{k}^\top B_{k}^\top}{{s}_{k}^\top B_{k}{s}_{k}}} \right),
\]
where $\tilde{l}_k = M s_k + \delta s_k = M^{\delta} s_k$, and $\{c_k\}$ is a sequence such that $c_k \in [0,1],$ and $\sum\limits_{k=0}^{\infty} c_k < \infty$.
We can rewrite the update formula \eqref{cond2} as
\begin{equation*}
B_{k+1} = B_k + c_k( B^{\rm BFGS}_{k+1} - B_k),
\end{equation*}
where $B^{\rm BFGS}_{k+1} $ is updated by the pure BFGS update \eqref{bfgsb} with respect to $M^{\delta}$ at every iteration.
Note that $B_{k+1} = B^{\rm BFGS}_{k+1}$ when $c_k = 1$.

We then propose the following construction of $T_k$ via the BFGS update.
\IncMargin{1em}
\begin{algorithm}
    Let $\delta \in (0, \infty)$, $\tau \in (\frac3 4, 1)$ and $B_0 \succeq \tau M$\; 
    Let $c_k$ be a sequence such that $c_k \in [0,1]$ and $\sum\limits_{k=0}^{\infty} c_k < \infty$\;
    If $s_k \neq 0$, then set $\tilde{l}_k = M^{\delta} s_k$ and update $B_{k+1}$ via
    \begin{equation*}
    \vspace{-3mm}
    B_{k+1}=B_{k}+ c_k \left(\frac {{\tilde{l}}_{k}{\tilde{l}}_{k}^\top}{{\tilde{l}}_{k}^\top {s}_{k}} -{\frac {B_{k}{s}_{k}{s}_{k}^\top B_{k}^\top}{{s}_{k}^\top B_{k}{s}_{k}}} \right);
    \end{equation*}\\
    Otherwise $$B_{k+1} = B_k;$$\\
    Construct $T_{k+1}$ as
\begin{equation*}T_{k+1} = B_{k+1} - M.\end{equation*}
\vspace{-3mm}
   \NoCaptionOfAlgo
   \caption{Construction of $T_k$ via the BFGS update\label{construction-t}}
\end{algorithm}

\subsection{Discussion on the Condition \ref{cond-main} for the indefinite matrix \texorpdfstring{$T_k$}{TEXT}}
We now consider matrices $\{T_+^k\}$ and $T_-$ such that
$T_k = T_+^k - T_-, \;\; T_+^k \succeq 0, \;\; T_- \succeq 0$
in Condition \ref{cond-main} (c).
Let
\begin{equation*}
T_+^k = B_k - \tau M \;\; \mathrm{and}\;\; T_- = (1-\tau)M, \;\; \mathrm{with} \;\; \tau \in [0,1).
\end{equation*}
Note that $T_k = T_+^k - T_- = B_k - M$ and $T_- \succeq 0$. Thus we only show that $T_+^k$ is positive semidefinite.

To this end, we give an extension result related to Theorem 2.2 in \cite{Gu2019bfgs}.

\begin{lemma}
\label{theorem:basic2}
Let $M \in \R^{n\times n}$ be a positive definite matrix. Let $s \in \R^n$ such that $s\neq 0$, and let $l=M s$.
If a given matrix $H_k \in \R^{n \times n}$ satisfies $H_k \preceq \tau_1 M^{-1}$ with $\tau_1 \geq 1$, then $H_{k+1}^{\mathrm{BFGS}}$ which is generated by the BFGS update \eqref{bfgsh} with respect to $M$ also satisfies $H_{k+1}^{\mathrm{BFGS}} \preceq \tau_1 M^{-1}$.
\end{lemma}
\begin{proof}
Let $v$ be an arbitrary nonzero vector in $\R^n$, and $\Psi=\{ z\in \R^n\;|\; s^\top z =0\}$.
As shown in \cite[Lemma 2.1]{Gu2019bfgs}, there exist $c \in \R$ and $z\in \Psi$ such that
$v=cl+z$. Together with $H_{k+1}^{\rm BFGS} l =s= M^{-1} l$ and $s^\top z = 0$, we can obtain that for any $\tau_1 \geq 1$,
\begin{eqnarray*}
v^\top H_{k+1}^{\rm BFGS} v &=& (cl+z)^\top H_{k+1}^{\rm BFGS} (cl+z) \\
&=& c^2 l^\top s + 2 cs^\top z + z^\top  H_{k+1}^{\rm BFGS} z \\
&=& c^2 l^\top  M^{-1} l+ z^\top  H_{k+1}^{\rm BFGS} z \\
&=& c^2 l^\top  M^{-1} l+ z^\top H_k z - 2 z^\top \left( \frac{s l^\top}{s^\top l} H_k\right) z + z^\top \left( \frac{s l^\top}{s^\top l} H_k \frac{l s^\top}{s^\top l}\right) z + \frac{z^\top s s^\top z}{s^\top l}  \\
&=& c^2 l^\top   M^{-1} l+ z^\top  H_k z \\
&\leq &c^2 l^\top \tau_1 M^{-1} l+ z^\top \tau_1 M^{-1} z \\
&=& (cl+z)^\top  \tau_1 M^{-1} (cl+z) - 2 \tau_1 cl^\top M^{-1} z \\
&=& v^\top   \tau_1 M^{-1} v,
\end{eqnarray*}
where the forth equality follows from \eqref{bfgsh}, and the inequality follows from the positive definiteness of $M^{-1}$ and the assumption that $H_k \preceq \tau_1 M^{-1}$.
Since $v$ is arbitrary, we have $H_{k+1}^{\rm BFGS} \preceq \tau_1 M^{-1}$.
\end{proof}
Lemma \ref{theorem:basic2} implies that
$B^{\rm BFGS}_{k+1} \succeq \tau M^\delta$ when $B_k \succeq \tau M^\delta$ with $\tau= \frac{1}{\tau_1} \leq 1$, and hence \begin{equation*}
B_{k+1} = (1-c_k) B_k + c_k B^{\rm BFGS}_{k+1} \succeq \tau M^\delta.
\end{equation*}
 That is, if $B_0 \succeq \tau M^\delta$ and $\tau \leq 1$, we have $B_k \succeq \tau M^\delta$ for all $k$, and hence $T_+^k \succeq 0$ for all $k$. When $\tau = 1$, it is reduced to the variable metric semi-proximal ADMM in \cite{Gu2019bfgs}.

For instance, we can choose the initial matrix $B_0$ as
\begin{equation*}
 B_0 = \xi I, \ \mbox{with}\ \xi = \tau \lambda_{\max}(M^{\delta}), \;
 \; \tau \in (0,1).
\end{equation*}
It is easy to see that $B_0 \succeq \tau M^{\delta}$.

Next we show that the $T_k$, $T_+^k$ and $T_-$ satisfy Condition \ref{cond-main} (d)-(g). We suppose that  $B_0 \succeq \tau M^{\delta}$ and $\tau \in (\frac3 4, 1)$.

First we show {\color{blue}Condition \ref{cond-main} (e)}. Note that ${s}_{k}^\top B_{k}{s}_{k} \geq \tau {s}_{k}^\top M^\delta {s}_{k} \geq \tau\delta \|s_k\|^2$, $\tilde{l}_k^\top s_k = s_k^\top M s_k + \delta \|s_k\|^2 \geq \delta \|s_k\|^2$, and $M$ is the constant matrix. Therefore, we can suppose that $\|B^{\rm BFGS}_{k+1} - B_k\|$ is bounded above by some constant $Q>0$, that is, $-QI \preceq B^{\rm BFGS}_{k+1} - B_k \preceq QI$. Moreover, $T_+^k = B_k - \tau M \succeq \tau M^{\delta} - \tau M \succeq \tau \delta I$. Then we can obtain that
\begin{align*}
T_+^{k+1} ={}& B_{k+1} - \tau M \nn \\
={} & B_k + c_k( B^{\rm BFGS}_{k+1} - B_k) - \tau M\nn \\
={} & T_+^k + c_k( B^{\rm BFGS}_{k+1} - B_k)  \nn \\
\preceq{} & T_+^k + \frac {c_k Q} {\tau \delta} \tau \delta I \nn \\
\preceq{} & T_+^k + \frac {c_k Q} {\tau \delta} T_+^k \nn \\
={} & (1 + \frac {c_k Q} {\tau \delta}) T_+^k.
\end{align*}
On the other hand, we have
\begin{align*}
T_+^{k} ={} & T_+^{k+1} - c_k( B^{\rm BFGS}_{k+1} - B_k)  \nn \\
\preceq{} & T_+^{k+1} + \frac {c_k Q} {\tau \delta} \tau \delta I \nn \\
\preceq{} & T_+^{k+1} + \frac {c_k Q} {\tau \delta} T_+^{k+1} \nn \\
={} & (1 + \frac {c_k Q} {\tau \delta}) T_+^{k+1}.
\end{align*}
Let $\gamma_k = \frac {Q}{\tau \delta} c_k$. Then we have \[\label{cond-e}\frac{1}{1+\gamma_k} T_+^k \preceq T_+^{k+1} \preceq (1+\gamma_k) T_+^k \;\;\mathrm{for}\;\mathrm{all}\;\; k.\]

Note that $\bar{M} = \nabla^2_{yy} g(y) = \Sigma_g$. Then $T_k + \Sigma_g + \beta B^\top B = B_k - M + \Sigma_g + \beta B^\top B = B_k \succ 0$ which shows that {\color{blue}Condition (d)} holds.

Next we show {\color{blue}Condition (f)}.
Since \eqref{cond-e} implies that 
$B_{k+1} - \tau M = T_+^{k+1} \preceq (1+\gamma_k)T_+^k = (1+\gamma_k)(B_k -\tau M)$ and $M$ is positive semidefinite, we have
$$B_{k+1} \preceq (1+\gamma_k)B_k - \gamma_k \tau M \preceq (1+\gamma_k)B_k.$$

Obviously,
$$ T_{k+1} + \Sigma_g + \beta B^\top B = B_{k+1} \preceq (1+\gamma_k)B_k = (1+\gamma_k)(T_k + \Sigma_g + \beta B^\top B).$$

Finally, we show {\color{blue}Condition (g)}. From the definition of $M$, we have
\begin{align*}
&{} T_k + \frac3 2 \Sigma_g - \frac{\gamma_{k-1}}{2} T_+^k - 2T_- + (\frac34 - \frac12 c)\beta B^\top B \nn \\
&{} = B_k - M + \frac3 2 \bar{M} - \frac{\gamma_{k-1}}{2} (B_k - \tau M) - 2(1-\tau)M +  (\frac34 - \frac12 c)\beta B^\top B \nn \\
&{} = \left(1-\frac{\gamma_{k-1}}{2}\right) B_k + \frac3 2 \bar{M} -M + \frac{\gamma_{k-1}}{2} \tau M - 2(1-\tau)M +  (\frac34 - \frac12 c)\beta B^\top B \nn \\
&{} \succeq \left(1-\frac{\gamma_{k-1}}{2}\right) \tau M + \frac3 2 \bar{M} -M + \frac{\gamma_{k-1}}{2} \tau M - 2(1-\tau)M +  (\frac34 - \frac12 c)\beta B^\top B \nn \\
&{} =  (3\tau - 3)(\bar{M} + \beta B^\top B) + \frac3 2 \bar{M} + (\frac34 - \frac12 c)\beta B^\top B \nn \\
&{} = (3\tau - \frac 3 2)\bar{M} + (3\tau - \frac94 - \frac12 c)\beta B^\top B,
\end{align*}
where the matrix inequality follows from $B_k \succeq \tau M^{\delta} = \tau M + \tau \delta I \succeq \tau M$. Note that there exist $\bar{k}$ such that $\gamma_k \leq 1$ for all $k \geq \bar{k}$. Without loss of generality, we assume $\bar{k} = 0$ and thus $\left(1- \frac{\gamma_{k-1}}{2}\right) \geq 0$ for all $k$.

Let $c = 2(\tau-\frac34)$. It is easy to see that $c \in (0, \frac12)$. Moreover, $3\tau - \frac32 > 0$ and $3\tau - \frac94 - \frac12 c = 2\tau - \frac3 2 >0$.

As a conclusion of the above discussion, the indefinite proximal term $T_k$ generated via the BFGS update can satisfy Condition \ref{cond-main}. Obviously, the VMIP-ADMM can cover the general indefinite proximal ADMM as the following remark.
\begin{remark}
When $\{T_k\}$  be a constant sequence for all $k$, that is, $T_k = T$, then we can write $T = T_+ - T_-$, where $T_+, T_- \succeq 0$. It is easy to check that the boundness Condition (e) and (f) immediately hold when $\gamma_k \equiv 0$. Let $T_+ = \tau(rI - \beta B^\top B) \succ 0$ and $T_- = (1 - \tau) \beta B^\top B \succeq 0$, we choose
\begin{equation*}
T = T_{+} - T_{-} =  \tau r I - \beta B^{\Tsf} B, ~~\mathrm{with}~~ r > \beta\|B^T B\|.
\end{equation*}
Condition (d) holds. For $\tau \in (0.75, 1)$, taking $c = 2(\tau-\frac34)$, then Condition (g) turns to be
\begin{align*}
T + \frac3 2 \Sigma_g - 2T_- + (\frac34 - \frac12 c)\beta B^\top B &{} \succ \tau \beta B^\top B - \beta B^\top B - 2(1-\tau)\beta B^\top B + (\frac34 - \frac12 c)\beta B^\top B \nn \\
&{} = (3\tau - \frac94 - \frac12 c)\beta B^\top B \nn \\
&{} \succ 0.
\end{align*}
It is reduced to the indefinite proximal ADMM in \cite{he2017optimal}.
\end{remark}

\section{Conclusions\label{sec:concl}}
In this paper, we proposed a variable metric indefinite proximal ADMM whose indefinite proximal term can be chosen differently at every iterative step. We proved the global convergence of the proposed method under some requirements by applying an analysis technique in \cite{Gu2019indefinite}. Moreover, for a special problem whose $y$-subproblems are unconstrained quadratic programming problem, we proposed to construct the indefinite term $T_k$ via the BFGS update. We showed that such construction can satisfy the general convergent conditions. 

Note that a strictly contractive version of the original ADMM which is known as the Peaceman-Rachford splitting method (PRSM) sometimes performs better in numerical experiments with some penalty parameters \cite{he2014strictly}. An indefinite proximal version of the PRSM also has been studied by many researchers \cite{gao2018symmetric, jiang2018generalized}. A further extension is to consider the variable metric indefinite term for PRSM. We leave this topic as one of our future work.

On the other hand, how to choose an adjusted proximal term is important to design a more efficient algorithm. The BFGS update provides better performance for some special problems whose $y$-subproblem is quadratic problem. It is worth developing some efficient proximal term for a general nonlinear subproblem.

\bibliographystyle{siam}
\bibliography{ILBFGS}
\end{document}